\documentclass[12pt, onesided, reqno]{amsart}

\usepackage{amssymb}
\usepackage{amsmath}
\usepackage{color}
\usepackage{enumerate}
\usepackage{graphicx}

\evensidemargin\oddsidemargin

\begin{document}

\setcounter{page}{1}

\newtheorem{PROP}{Proposition}
\newtheorem{REMS}{Remark}
\newtheorem{LEM}{Lemma}
\newtheorem{THEA}{Theorem A\!\!}
\renewcommand{\theTHEA}{}
\newtheorem{THEB}{Theorem B\!\!\!}
\renewcommand{\theTHEB}{}

\newtheorem{theorem}{Theorem}
\newtheorem{proposition}[theorem]{Proposition}
\newtheorem{corollary}[theorem]{Corollary}
\newtheorem{lemma}[theorem]{Lemma}
\newtheorem{assumption}[theorem]{Assumption}

\newtheorem{definition}[theorem]{Definition}
\newtheorem{hypothesis}[theorem]{Hypothesis}

\theoremstyle{definition}
\newtheorem{example}[theorem]{Example}
\newtheorem{remark}[theorem]{Remark}
\newtheorem{question}[theorem]{Question}

\newcommand{\eqnsection}{
\renewcommand{\theequation}{\thesection.\arabic{equation}}
    \makeatletter
    \csname  @addtoreset\endcsname{equation}{section}
    \makeatother}
\eqnsection

\def\a{\alpha}
\def\b{\beta}
\def\B{{\bf B}} 
\def\cia{c_{\a, \infty}}
\def\coa{c_{\a, 0}}
\def\cua{c_{\a, u}}
\def\cL{{\mathcal{L}}} 
\def\Ea{E_\a}
\def\eps{\varepsilon}
\def\g{{\gamma}} 
\def\oga{{\overline{\gamma}}}
\def\Ga{{\Gamma}} 
\def\i{{\rm i}}
\def\K{{\bf K}}
\def\Ka{{\bf K}_\a}
\def\kar{{\kappa_{\a, \rho}}}
\def\L{{\bf L}}
\def\lbd{\lambda}
\def\lcr{\left[}
\def\lpa{\left(}
\def\lva{\left|}
\def\otheta{\overline{\theta}}
\def\rpa{\right)}
\def\rcr{\right]}
\def\rva{\right|}
\def\T{{\bf T}}
\def\M{{\mathcal M}}
\def\X{{\bf X}}
\def\Ton{{T_0^{(n)}}}
\def\Lton{{\vert L_{T_0^{(n)}}\vert}}
\def\Un{{\bf 1}}
\def\ZZ{{\bf Z}}
\def\CC{{\bf C}}
\def\GG{{\bf \Ga}}
\def\BB{{\bf B}}
\def\car{c_{\a,\rho}}
\def\sar{s_{\a,\rho}}
\def\pbxy{\Pb_{(x,y)}}
\def\foxy{f^0_{x,y}}

\def\E{\mathbb{E}}
\def\Z{\mathbb{Z}}
\def\N{\mathbb{N}}
\def\Q{\mathbb{Q}}
\def\R{\mathbb{R}}
\def\P{\mathcal{P}}
\def\Pb{\mathbb{P}}
\def\C{\mathbb{C}}
\def\cC{\mathcal{C}}
\def\F{\mathcal{F}}
\def\S{\mathcal{S}}
\def\W{\mathcal{W}}
\def\L{\mathcal{L}}
\def\G{\mathcal{G}}

\newcommand{\equi}{\mathop{\sim}\limits}
\def\={{\;\mathop{=}\limits^{\text{(law)}}\;}}
\def\d{{\;\mathop{=}\limits^{\text{(1.d)}}\;}}
\def\st{{\;\mathop{\geq}\limits^{\text{(st)}}\;}}

\def\cas{\stackrel{a.s.}{\longrightarrow}}
\def\claw{\stackrel{d}{\longrightarrow}}
\def\elaw{\stackrel{d}{=}}
\def\qed{\hfill$\square$}
                  
\title[Windings of the stable Kolmogorov process]
       {Windings of the stable Kolmogorov process}
\author[Christophe Profeta]{Christophe Profeta}

\address{ Laboratoire de Math\'ematiques et Mod\'elisation d'Evry (LaMME),   Universit\'e d'Evry-Val-d'Essonne, UMR CNRS 8071, F-91037 Evry Cedex. {\em Email} : {\tt christophe.profeta@univ-evry.fr}}

\author[Thomas Simon]{Thomas Simon}

\address{Laboratoire Paul Painlev\'e, Universit\'e Lille 1, F-59655 Villeneuve d'Ascq Cedex and Laboratoire de Physique Th\'eorique et Mod\`eles Statistiques, Universit\'e  Paris-Sud, F-91405 Orsay Cedex. {\em Email} : {\tt simon@math.univ-lille1.fr}}



\begin{abstract} We investigate the windings around the origin of the two-dimensional Markov process $(X,L)$ having the stable L\'evy process $L$ and its primitive $X$ as coordinates, in the non-trivial case when $\vert L\vert$ is not a subordinator. First, we show that these windings have an almost sure limit velocity, extending McKean's result \cite{McK} in the Brownian case. Second, we evaluate precisely the upper tails of the distribution of the half-winding times, connecting the results of our recent papers \cite{CP1, PS}.
\end{abstract}

\maketitle
 
\section{Introduction and statement of the results}

A celebrated theorem by F.~Spitzer \cite{Spi} states that the angular part $\{\omega (t), \, t\ge 0\}$ of a two-dimensional Brownian motion starting away from the origin satisfies the following limit theorem
$$\frac{2\,\omega (t)}{\log t}\; \claw\; \cC\qquad \mbox{as $t\to +\infty,$}$$
where $\cC$ denotes the standard Cauchy law. An analogue of this result for isotropic stable L\'evy processes was given in \cite{BW}, with a slower speed in $\sqrt{\log t}$  and a centered Gaussian limit law. Notice that both these results can be obtained as functional limit theorems with respect to the Skorohod topology. We refer to \cite{DV} for a recent paper revisiting these problems, with further results and an updated bibliography. \\

In a different direction, McKean \cite{McK} had observed that the windings of the Kolmogorov diffusion, which is the two-dimensional process $Z$ having a linear Brownian motion as second coordinate and its running integral as first coordinate, obey an almost sure limit theorem. More precisely, if $\{\omega (t), \, t\ge 0\}$ denotes the angular part of the  process $Z$ starting away from the origin, it is shown in Section 4.3 of \cite{McK} that
$$\frac{\omega (t)}{\log t}\; \cas\; -\frac{{\sqrt 3}}{2}\qquad \mbox{as $t\to +\infty$}$$
(the constant which is given in \cite{McK} is actually $-{\sqrt 3}/8,$ but it will be observed below that the evaluation of the relevant improper integral in \cite{McK} was slightly erroneous). Of course, the degeneracy of the Kolmogorov diffusion makes it wind in a very particular way, since this process visits a.s. alternatively and clockwise the left and right half-planes. The regularity of this behaviour, which contrasts sharply with the complexity of planar Brownian motion, makes it possible to use the law of large numbers and to get an almost sure limit theorem.\\

The first aim of this paper is to obtain an analogue of McKean's result in replacing Brownian motion by a strictly $\a-$stable L\'evy process $L = \{L_t, \, t\ge 0\}.$  Without loss of generality, we choose the following normalization for the characteristic exponent
\begin{equation}
\label{Norm}
\Psi(\lbd)\; =\;\log(\E[e^{\i \lbd L_1}])\; =\; -(\i \lbd)^\a e^{-\i\pi\a\rho\, {\rm sgn}(\lbd)}, \qquad \lbd\in\R,
\end{equation}
where $\a\in (0,2]$ is the self-similarity parameter and $\rho = \Pb[L_1 \ge 0]$ is the positivity parameter. We refer to \cite{ST, Z} for accounts on stable laws and processes, and to the introduction of our previous paper \cite{PS} for a discussion on this specific parametrization. Recall that if $\a = 2,$ then necessarily $\rho =1/2$ and $L= \{\sqrt{2} B_t, \, t\ge 0\}$ is a rescaled Brownian motion. Introduce the primitive process
$$X_t \; =\; \int_0^t L_s\, ds,\qquad t\ge 0,$$ 
and denote by $\Pb_{(x,y)}$ the law of the strong Markov process $Z = (X,L)$ started from $(x,y)$. By analogy with the classical Kolmogorov diffusion \cite{Ko}, this process may and will be called the stable Kolmogorov process. When $(x,y)\neq (0,0),$ it can be shown without much difficulty - see Lemma \ref{infinity} below - that under $\pbxy,$ the process $Z$ never hits $(0,0).$ Filling in the gaps made by the jumps of $L$ by vertical lines - see the figure below - and reasoning exactly as in \cite{BW} p.1270 it is possible to define the algebraic angle 
$$\omega(t)\;=\; \widehat{(Z_0, Z_t)}$$
measured in the trigonometric orientation. 

\begin{figure}[h]
\includegraphics[height=8cm]{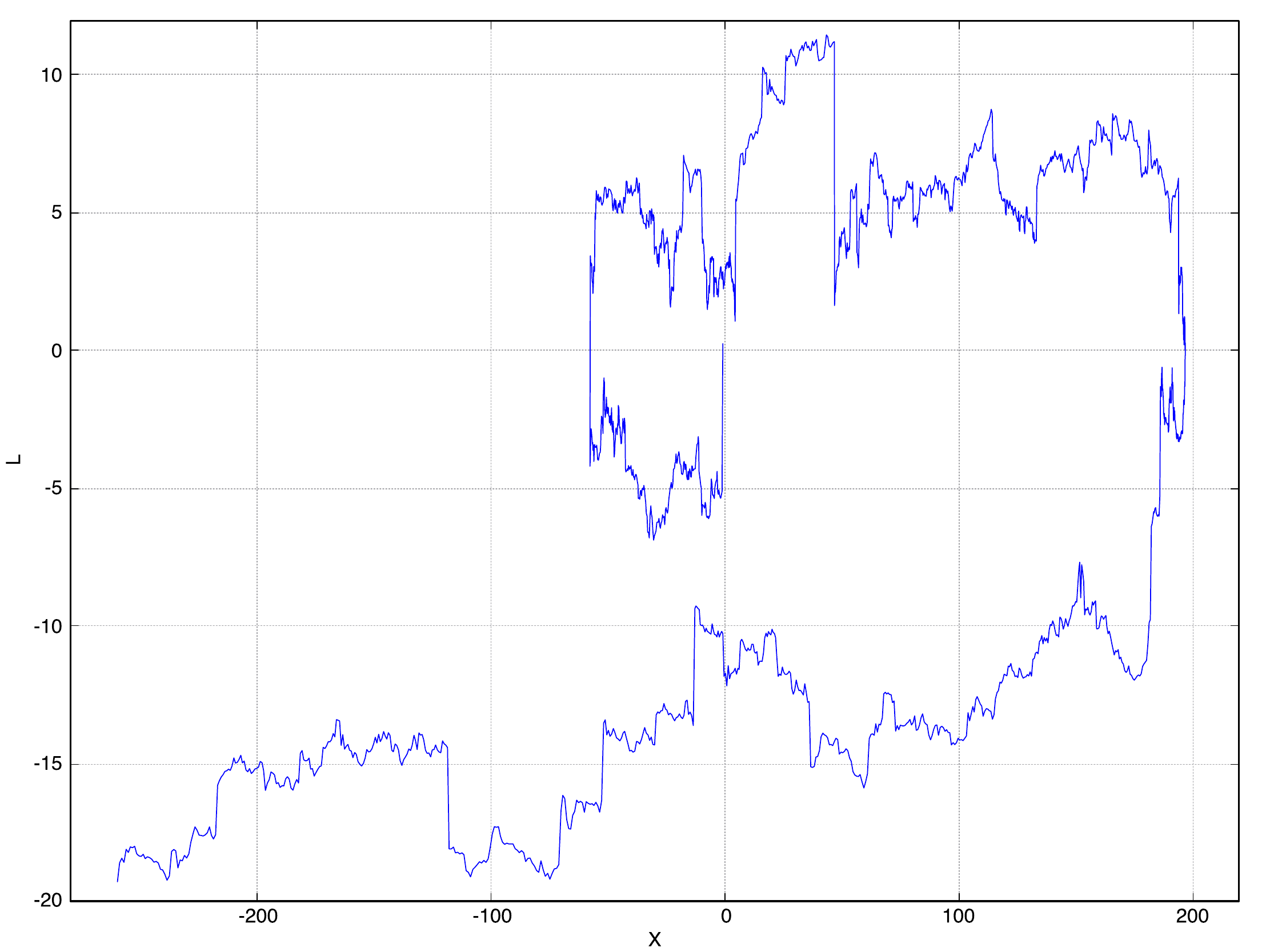}
\caption{One path of $((X_t,L_t), t\leq 100)$ starting from $X_0=-1$ and $L_0=0$.}
\end{figure}

\noindent
If $\rho = 1$ resp. $\rho =0$, then $\vert L\vert$ is a stable subordinator and it is easy to see that $Z$ stays for large times within the positive resp. the negative quadrant with a.s. $X_t/L_t \to +\infty,$ so that $\omega (t)$ converges a.s. to a finite limit which is
$$\widehat{(Z_0, {\rm O}x)}\qquad\mbox{resp.}\qquad\widehat{(-Z_0, {\rm O}x)}.$$
When $\rho\in (0,1)$ and $(x,y)\neq (0,0),$ the L\'evy process $L$ oscillates and the Kolmogorov process $Z$ winds clockwise and infinitely often around the origin as soon as $(x,y)\neq (0,0).$ Indeed, considering the partition $\R^2\,\backslash\{(0,0)\} = \P_-\cup \P_+$ with
$$\P_-\; =\; \{x<0\}\,\cup\,\{x=0, y<0\}\qquad\mbox{and}\qquad \P_+\; =\; \{x>0\}\,\cup\,\{x=0, y>0\},$$
we see that if $(x,y)\in\P_-$ the continuous process $X$ visits alternatively the negative and positive half-lines, starting negative, and that its speed when it hits zero is alternatively positive and negative, starting positive. When $(x,y)\in\P_+$ the same alternating scheme occurs, with opposite signs. In particular, the function $\omega(t)$ is a.s. negative for all $t$ large enough. In order to state our first result, which computes the a.s. limit velocity of $\omega(t),$ let us finally introduce the parameters
$$\displaystyle \gamma\; =\; \frac{\rho\a}{1+\a}\; \in\; (0,1/2) \qquad \text{ and }\qquad  \overline{\gamma}\; =\; \frac{(1-\rho)\a}{1+\a}\; \in\; (0,1/2).$$ 

\begin{THEA}
Assume $\rho\in (0,1)$ and $(x,y)\neq(0,0)$. Then, under $\pbxy,$ one has 
$$\frac{\omega(t)}{\log t}\; \cas\; - \frac{2\sin(\pi\gamma)\sin(\pi \overline{\gamma})}{\a\sin(\pi (\gamma+\overline{\gamma}))}  \qquad \mbox{as $t\rightarrow +\infty.$}$$ 
\end{THEA}

\noindent
Note that in the Brownian motion case $\alpha=2$, we have $\rho=1/2$ and $\gamma=\overline{\gamma}=1/3,$ so that 
 $$\displaystyle \frac{2\sin(\pi\gamma)\sin(\pi \overline{\gamma})}{\a\sin(\pi (\gamma+\overline{\gamma}))}\;=\;\frac{\sqrt{3}}{2}\cdot$$  
 The constant $-\sqrt{3}/8$ which is given in Section 4.3 of \cite{McK} is not the right one because of the erroneous evaluation of the integral in (3.8.a) therein: this integral equals actually $\pi/\sqrt{3},$ as can be checked by an appropriate contour integration. The proof of Theorem A goes basically along the same lines as in \cite{McK}. We consider the successive hitting times of $0$ for the integrated process $X:$ 
$$T_0^{(1)}\; =\; T_0\;=\;\inf\{t>0,\; X_t=0\} \qquad \mbox{and}\qquad \Ton\;=\;\inf\{t>T_0^{(n-1)},\; X_t=0\},$$
which can be viewed as the half-winding times of $Z.$ We first check that $n\mapsto\Ton$ increases a.s. to  $+\infty$ as soon as $(x,y)\neq (0,0).$ The exact exponential rate of escape of $\Ton,$ which yields the exact winding velocity, is computed thanks to an elementary large deviation argument involving the law of $L_{T_0}$ under $\Pb_{(x,0)},$ a certain transform of the half-Cauchy distribution as observed in \cite{PS}. Notice that contrary to \cite{McK} where the proof is only sketched, we provide here an argument with complete details.\\

In the Brownian case $\a = 2,$ an expression of the law of the bivariate random variable $(\Ton, \vert L_\Ton\vert)$ under $\Pb_{(0,y)}$ has been given in Theorem 1 of \cite{La2}, in terms of the modified Bessel function of the first kind. This expression becomes very complicated under $\pbxy$ when $x\neq 0,$ even for $n = 1$ - see Formula (2) p.4 in \cite{La2}. In all cases, this expression is not informative enough to evaluate the upper tails of $\Ton.$ In \cite{CP1} it was shown that 
$$\pbxy[ \Ton \ge t] \; \asymp\;t^{-1/4}\; (\log t)^{n-1}\qquad \mbox{as $t\rightarrow  +\infty$}$$
where, here and throughout, the notation $f(t)\asymp g(t)$ means that there exist two constants $0<\kappa_1 \leq \kappa_2 <+\infty$ such that $\kappa_1 f(t) \leq g(t) \leq \kappa_2 f(t)$ as $t\rightarrow +\infty$. On the other hand, Theorem A in our previous paper \cite{PS} shows the non-trivial asymptotics 
$$ \Pb_{(x,y)}[T_0>t]\;\asymp\;  t^{-\theta} \qquad \mbox{as $t\rightarrow  +\infty$}$$
for $(x,y)\in\P_-,$ with $\theta= \rho/(1+\alpha(1-\rho)).$ By symmetry, the latter result also shows that $ \Pb_{(x,y)}[T_0>t]\asymp t^{-\otheta}$ as $t\rightarrow  +\infty$ for $(x,y)\in\P_+,$ with $\otheta= (1-\rho)/(1+\alpha\rho).$ The second main result of this paper connects the two above estimates.

\begin{THEB}
Assume that $\rho\in (0,1)$ and $(x,y)\in\P_-$. For every $n\geq 2,$ the following asymptotics hold as $t\to +\infty:$
$$\left\{\begin{array}{ll}
\Pb_{(x,y)}[\Ton>t]\,\asymp \,t^{-\theta}\; (\log t)^{\left[\frac{n-1}{2}\right]} & \mbox{if $\rho < 1/2,$}\\
\vspace{-.2cm}\\
\Pb_{(x,y)}[\Ton>t]\,\asymp \,t^{-\otheta}\; (\log t)^{\left[\frac{n}{2}\right]-1} & \mbox{if $\rho > 1/2,$}\\
\vspace{-.2cm}\\
\Pb_{(x,y)}[\Ton>t]\,\asymp \,t^{-\theta}\; (\log t)^{n-1} & \mbox{if $\rho =1/2.$}
\end{array}\right.$$
\end{THEB}
By symmetry, the same result holds for $(x,y)\in\P_+$ with $\theta$ and $\otheta$ switched. In the above statement, the separation of cases is intuitively clear, since for $\rho<1/2$ resp. $\rho > 1/2$ the negative resp. positive excursions below 0 will tend to prevail. The main difficulty in the proof of Theorem B is to show that the required asymptotic behaviour does not depend on the starting point $(x,y)\in\P_-.$ This is handled thanks to a uniform estimate on the Mellin transform of the harmonic measure $\pbxy[L_{T_0}\in .],$ and a general estimate on the upper tails of the product of two positive independent random variables. These two estimates have both independent interest.

\section{Proofs}

\subsection{Preliminary results}

As mentioned before, we first establish some estimates on the harmonic measure of the left half-plane with respect to the stable Kolmogorov process. When starting from $(x,y)\in\P_-,$ the process $(X,L)$ ends up in exiting $\P_-$ on the positive vertical axis, and its exit distribution is given by the law of $L_{T_0}$ under $\Pb_{(x,y)}$. This distribution is called the harmonic measure since by the generalized Poisson formula, it allows to construct harmonic functions with respect to the degenerate operator
$$\cL^{\a, \rho}_y\; +\; y \, \frac{\partial}{\partial x}$$
on the half-plane, where $\cL^{\a, \rho}_y$ is the generator of the stable L\'evy process $L.$ However, we shall not pursue these lines of research here.

\begin{lemma}\label{lem:asympLT0}
Assume that $(x,y)\in\P_-.$ The Mellin transform $s\mapsto \E_{(x,y)} [L_{T_0}^{s-1}]$ is real-analytic on $(1/(\g-1), 1/(1-\g)),$ with two simple poles at $1/(\g -1)$ and $1/(1-\g).$ In particular, the random variable $L_{T_0}$ has a smooth density $f^0_{x,y}$ under $\pbxy,$ and there exist $c_1, c_2 >0$ such that 
$$f^0_{x,y}(z)\, \equi_{z\rightarrow0} \, c_1\, z^{\a\theta/\gamma}\qquad \text{and}\qquad f^0_{x,y}(z)\, \equi_{z\rightarrow+\infty}\,  c_2\,  z^{-\alpha \theta-1}.$$
\end{lemma}
\begin{proof}
Observe first that the smoothness and the asymptotic behaviour of the density function of $L_{T_0}$ are a direct consequence of the statement on the Mellin transform, thanks to the converse mapping theorem stated e.g. as Theorem 4 in \cite{FGD}. This latter statement is also a direct consequence of Theorem B in \cite{PS} when either $x=0$ or $y=0$. From now on we shall therefore assume that $xy \neq 0$. By Proposition 2 (i) and Equation (3.2) in \cite{PS} we have
\begin{equation}\label{eq:MelLap}
\E_{(x,y)}\left[L_{T_0}^{s-1}\right]\; = \; \frac{\pi {\displaystyle \int_0^{+\infty} \E_{(x,y)}\left[X_t^{-\nu}\Un_{\{X_t>0\}}\right]dt } }{(1+\alpha)^{1-\nu}\left(\Gamma\left(1-\nu\right)\right)^2\Gamma(1-s)\sin(\pi s (1-\gamma))}
\end{equation}
with $s=(1-\nu)(1+\alpha)\in(0,1)$. However, it does not seem easy to study the poles of the right-hand side directly since the integral is not expressed in closed form, and for this reason we shall perform a further Mellin transformation in space. First, we know from Proposition 1 in \cite{PS} that
$$\E_{(x,y)}[X_t^{-\nu}\Un_{\{X_t>0\}}] = \frac{\Gamma(1-\nu)}{\pi}
 \int_0^{\infty} \! \lbd^{\nu-1} e^{-\car\lbd^\alpha t^{\alpha +1}} \sin(\lambda(x+yt) + \sar\lambda^\alpha t^{\alpha+1} + \pi\nu/2)\, d\lambda$$
for every $\nu\in (0,1),$  with $$\sar\, = \,\frac{\sin (\pi\a(\rho -1/2))}{\a +1}\, \in \, (-1,1)\qquad\mbox{and}\qquad \car\, =\, \frac{\cos(\pi\a (\rho -1/2))}{\a +1}\, \in\, (0,1).$$
For every $\beta\in (0,\nu)$ this yields
\begin{align*}
&\int_{-\infty}^0 |x|^{\beta-1} \E_{(x,y)}[X_t^{-\nu}\Un_{\{X_t>0\}}] dx\\
&\qquad=\;\frac{\Gamma(1-\nu)}{\pi}\int_0^{+\infty} \lambda^{\nu-1} e^{-\car\lbd^\alpha t^{\alpha +1}} \int_{-\infty}^0  |x|^{\beta-1}\sin(\lambda(x+yt) + \sar\lambda^\alpha t^{\alpha+1} + \pi\nu/2)\,dx\,  d\lambda\\
&\qquad =\;\frac{\Gamma(1-\nu)\Gamma(\beta)}{\pi}\int_0^{+\infty} \lambda^{\nu-\beta-1} e^{-\car\lbd^\alpha t^{\alpha +1}} \sin(\lambda yt + \sar\lambda^\alpha t^{\alpha+1} + \pi(\nu-\beta)/2)\,  d\lambda\\
&\qquad =\;\frac{\Gamma(1-\nu)\Gamma(\beta)}{\Gamma(1-\nu+\beta)}\; \E_{(0,y)}\left[X_t^{-\nu+\beta}\Un_{\{X_t>0\}}\right]
\end{align*}
where the switching of the first equality is justified exactly as in Lemma 1 of \cite{PS}, and the second equality follows from trigonometry and generalized Fresnel integrals - see (2.1) and (2.2) in \cite{PS}. Assume first that $y<0$. From Proposition 2 (ii) in \cite{PS}, we obtain
\begin{multline*}\int_0^{+\infty} \int_{-\infty}^0 |x|^{\beta-1} \E_{(x,y)}[X_t^{-\nu}\Un_{\{X_t>0\}}] \,dx\,dt \\
=\, \frac{\Gamma(1-\nu)\Gamma(\beta)\Gamma(1-\nu+\beta)}{\pi} \,(\alpha+1)^{1-\nu+\beta} \,\Gamma(1-s-\beta(1+\alpha))\,  \sin(\pi \rho \alpha \beta + \pi\gamma s)\, |y|^{s+\beta(1+\alpha)-1}
\end{multline*}
for every $\beta\in (0, (1-s)/(\a+1)).$ 
Putting this together with (\ref{eq:MelLap}), we finally deduce
\begin{multline}\label{eq:Melx2} \int_{-\infty}^0 |x|^{\beta-1} \E_{(x,y)}\left[L_{T_0}^{s-1}\right]   dx \\
= \;(\alpha+1)^{\beta}\, \frac{\Gamma(\beta)\sin(\pi \rho \alpha \beta + \pi\gamma s)\Gamma(1-\nu+\beta)\Gamma(1-s-\beta(1+\alpha))}{\Gamma(1-\nu)\Gamma(1-s)\sin(\pi s(1-\gamma))}\, |y|^{s+\beta(1+\alpha)-1}.
\end{multline}
We shall now invert this Mellin transform in the variable $\beta$ in order to get a  suitable integral expression for $\E_{(x,y)}[L_{T_0}^{s-1}].$ Fix $\beta\in (0, (1-s)/(\a+1)).$  On the one hand, since $\rho\alpha <1$, we have 
$$\Gamma(\beta)\cos((\pi \rho \alpha \beta + \pi\gamma s)/2)\;= \; \int_0^{\infty} x^{\beta-1}\,  e^{-x \cos(\pi\rho \alpha/2)} \cos(x\sin(\pi\rho \alpha/2)  +\pi \gamma s/2)\,dx$$
and
\begin{align*}
\Gamma(1-\nu+\beta)\sin((\pi \rho \alpha \beta + \pi\gamma s)/2)& = \Gamma(1-\nu+\beta) \sin\left(\frac{\pi \rho \alpha}{2} (\beta + 1-\nu)\right)\\
&=  \int_0^{\infty} x^{\beta-1}\,  x^{1-\nu}e^{-x \cos(\pi\rho \alpha/2)} \sin(x\sin(\pi\rho \alpha/2))\,dx.
\end{align*}
On the other hand, a change of variable in the definition of the Gamma function shows that
$$\Gamma(1-s-\beta(1+\alpha)) =\frac{1}{1+\alpha} \int_0^{+\infty} x^{\beta-1} x^{\frac{s-1}{1+\alpha}} e^{-x^{-1/(1+\alpha)}} dx.$$
Setting
$$K_s(\xi) = \int_0^{+\infty} z^{\frac{s}{1+\alpha}-1}  e^{-\cos(\pi\rho \alpha/2)(\xi/z+z) } \cos\left(\frac{\xi}{z}\sin(\pi\rho \alpha/2)  +\pi \gamma s/2\right)  \sin(z\sin(\pi\rho \alpha/2))\,dz,$$
\noindent
we can now invert (\ref{eq:Melx2}) and obtain, applying Fubini's theorem and using the notation ${\widetilde x} = x (1+\a)\vert y\vert^{1+\a},$ a new expression for the Mellin transform of $L_{T_0}$ :
$$\E_{(\widetilde{x}, y)}\left[L_{T_0}^{s-1}\right] = \frac{2 |y|^{s-1} |x|^{\frac{s-1}{1+\alpha}}(1-s)}{\Gamma(s/(1+\alpha)) \Gamma(2-s)\sin(\pi s (1-\gamma))}  \int_0^{+\infty}\xi^{\frac{1-s}{1+\alpha}-1} e^{-\left(\frac{\xi}{|x|}\right)^{1/(1+\alpha)}} K_s\left(\xi\right) \,d\xi.$$
Since $\frac{1}{1-\gamma} = 1+\frac{\alpha \rho}{\alpha+(1-\alpha \rho)}<1+\rho<2$, it remains to prove that the function 
$$H_s(x) = (1-s)\int_0^{\infty}\xi^{\frac{1-s}{1+\alpha}-1} e^{-\left(\frac{\xi}{|x|}\right)^{1/(1+\alpha)}} K_s\left(\xi\right) \,d\xi$$
admits an analytic continuation on $[1/(\g-1), 1/(1-\g)]$. Observe first that for any $s>-1-\alpha$, the function $K_s$  is uniformly bounded on $[0,+\infty)$ by 
$$|K_s(\xi)| \leq \sin(\pi\rho \alpha/2) \int_0^{+\infty} z^{\frac{s}{1+\alpha}}  e^{-\cos(\pi\rho \alpha/2)z}\,dz = \frac{\sin(\pi\rho \alpha/2)}{\left(\cos(\pi\rho \alpha/2)\right)^{\frac{s}{1+\alpha}+1}}\Gamma\left(\frac{s}{1+\alpha}+1\right).$$
As a consequence, the function $H_s$ has an analytic continuation on $(-1-\alpha, 1)\supset [1/(\gamma-1),1)$ since $\frac{1}{\g-1}=-\frac{1+\alpha}{1+\alpha(1-\rho)}>-1-\alpha$. Next for every $s\in (0,1)$ an integration by parts shows that
\begin{align*}
H_s(x)&=(1+\alpha)\int_0^{\infty}  \xi^{\frac{1-s}{1+\alpha}} \, \frac{d}{d\xi}  \left(e^{-\left(\frac{\xi}{|x|}\right)^{1/(1+\alpha)}} K_s\left(\xi\right)\right) \,d\xi\\
&=(1+\alpha)\int_0^{\infty}  \xi^{\frac{1-s}{1+\alpha}}e^{-\left(\frac{\xi}{|x|}\right)^{\frac{1}{1+\alpha}}} K_s^\prime\left(\xi\right) \,d\xi - \frac{1+\alpha}{|x|^{1/(1+\alpha)}}\int_0^{+\infty}  \xi^{\frac{1-s-\alpha}{1+\alpha}} e^{-\left(\frac{\xi}{|x|}\right)^{\frac{1}{1+\alpha}}} K_s\left(\xi\right) \,d\xi,
\end{align*}
where $K_s^\prime$ is well-defined on $[0,+\infty)$ for any $s> 0$, and bounded by 
$$
|K_s^\prime(\xi)| \leq  2\sin(\pi\rho \alpha/2) \int_0^{\infty}z^{\frac{s}{1+\alpha}-1}  e^{-\cos(\pi\rho \alpha/2)z } dz\;=\;\frac{2\sin(\pi\rho \alpha/2)}{\left(\cos(\pi\rho \alpha/2)\right)^{s/(1+\alpha)}} \Gamma\left(\frac{s}{1+\alpha}\right).
$$
Consequently, the function $H_s$ also admits an analytic continuation on $(0,2)\supset (0, 1/(1-\gamma)]$. This completes the proof in the case $y <0.$ The case $y>0$ may be dealt with in an entirely similar way, and we leave the details to the reader.

\end{proof}

Our second preliminary result is elementary, but we could not find any reference in the literature and we hence provide a proof.

\begin{lemma}\label{lem:ln}
Let $\mu \geq \nu >0$ and $n,p \in \N$. Assume that $X$ and $Y$ are two independent positive random variables such that :
$$\Pb[X\geq z] \mathop{\asymp}\limits_{z\rightarrow +\infty} z^{-\nu}(\log z)^n\qquad \text{ and }\qquad \Pb[Y\geq z] \mathop{\asymp}\limits_{z\rightarrow +\infty}  z^{-\mu}(\log z)^p.$$
Then
$$\left\{\begin{array}{ll}
 \Pb[XY\geq z] \mathop{\asymp}\limits_{z\rightarrow +\infty}  z^{-\nu} (\log z)^{n+p+1}& \mbox{if $\mu=\nu,$}\\
\vspace{-.2cm}\\
\Pb[XY\geq z] \mathop{\asymp}\limits_{z\rightarrow +\infty}  z^{-\nu} (\log z)^{n} & \mbox{if $\mu>\nu.$}
\end{array}\right.$$
\end{lemma}

\begin{proof}
We first decompose the product as 
$$\Pb[XY\geq z]\; =\; \int_0^{\infty}  \Pb [X\geq zy^{-1}] \,\Pb[Y\in dy].$$
Therefore, for $z>A$ large enough, 
\begin{align*}
\Pb[XY\geq z]  &\geq \;\int_A^{\sqrt{ z}} \Pb [X\geq zy^{-1}] \,\Pb[Y\in dy]\\
&\geq\; \frac{\kappa_1 }{z^\nu}  \int_A^{\sqrt{z}}  y^\nu(\log(zy^{-1}))^n  \,\Pb[Y\in dy]\;\geq \; \frac{\kappa_1 (\log(z))^n}{2^n z^\nu}  \int_A^{\sqrt{z}}  y^\nu\,  \Pb[Y\in dy].
\end{align*}
Then, integrating by parts,
$$\int_A^{\sqrt{z}}  y^\nu\, \Pb[Y\in dy]\;=\; A^\nu\Pb[Y\geq A]  -z^{\nu/2}  \Pb[Y\geq \sqrt{z}]\; +\; \nu \int_A^{ \sqrt{z}} y^{\nu-1}\, \Pb[Y\geq y] dy.$$
Now, if $\nu<\mu$, this expression remains bounded as $z\rightarrow +\infty$. Assume therefore that $\mu=\nu$. In this case, we have :
$$\int_A^{\sqrt{z}}  y^\nu \,\Pb[Y\in dy]\geq A^\nu\Pb[Y\geq A]\, -\, \frac{\kappa_2}{2^p} (\log z)^p \,+\, \nu \kappa_1 \int_A^{ \sqrt{z}} \frac{(\log y)^p}{y} dy \;\equi_{z\rightarrow +\infty}\; \frac{\nu\kappa_1\,(\log z)^{p+1}}{2^{p+1}(p+1)}, $$
which gives the lower bound. To obtain the upper bound, we separate the integral in three parts and proceed similarly, with $\varepsilon$ small enough:
\begin{multline*}
\int_0^{1/\varepsilon}  \Pb[X\geq zy^{-1}] \Pb[Y\in dy]\, +\, \int_{1/\varepsilon}^{\varepsilon z}  \Pb[X\geq zy^{-1}]  \Pb[Y\in dy]\, +\,\int_{ \varepsilon z}^{\infty}  \Pb[X\geq zy^{-1}]\Pb[Y\in dy]  \\
\leq\; \Pb[X\geq \varepsilon z] + \frac{\kappa_2 (\log(\varepsilon z))^{n}}{z^\nu}\int_{1/\varepsilon}^{ \varepsilon z}y^{\nu}  \Pb[Y\in dy]\,+\, \Pb[Y\geq \varepsilon z]
\end{multline*}
and the proof is concluded as before, using an integration by parts and looking separately at both cases $\nu<\mu$ and $\nu=\mu$.

\end{proof}

\subsection{Proof of Theorem A} 
By symmetry, it is enough to show Theorem A for $(x,y)\in\P_-.$ Consider the sequence
$$\lpa \Ton, \Lton\rpa_{n\ge 1}$$
and set $\{\F_n, n\ge 1\}$ for its natural completed filtration. It is easy to see from the strong Markov and scaling properties of $Z$ that this sequence is Markovian. To be more precise, starting from $\P_-$ and taking into account the possible asymmetry of the process $L$, we have the following identities for all $p\ge 1.$
$$\lpa T_0^{(2p)}, \vert L_{T_0^{(2p)}}\vert\rpa\; \elaw\;  \lpa T_0^{(2p-1)}+ \vert L_{T_0^{(2p-1)}}\vert^\a \tau^+, \vert L_{T_0^{(2p-1)}}\vert\ell^+\rpa$$
with $(\tau^+, \ell^+)\perp \F_{2p-1}$ distributed as $(T_0, \vert L_{T_0}\vert)$ under $\Pb_{(0,1)},$ and
$$\lpa T_0^{(2p+1)}, \vert L_{T_0^{(2p+1)}}\vert\rpa\; \elaw\;  \lpa T_0^{(2p)}+ \vert L_{T_0^{(2p)}}\vert^\a \tau^-, \vert L_{T_0^{(2p)}}\vert\ell^-\rpa$$
with $(\tau^-, \ell^-)\perp \F_{2p}$ distributed as $(T_0, \vert L_{T_0}\vert)$ under $\Pb_{(0,-1)}.$ The starting term $(T_0, \vert L_{T_0}\vert)$ has the same law as $(\tau^-, \ell^-)$ if $x =0$ and $y=-1.$ By induction we deduce the identities
$$\vert L_{T_0^{(2p)}}\vert\; \elaw\;  \vert L_{T_0}\vert\; \times\;  \prod_{k=1}^{p-1} \ell_k^-\;\times\;\prod_{k=1}^p \ell_k^+\qquad\mbox{and}\qquad\vert L_{T_0^{(2p+1)}}\vert\; \elaw\; \vert L_{T_0}\vert\; \times\;  \prod_{k=1}^p \ell_k^-\;\times\;\prod_{k=1}^p \ell_k^+$$
where, here and throughout, $(\tau^\pm_k, \ell^\pm_k)_{k\ge 1}$ are two i.i.d. sequences distributed as $(\tau^\pm, \ell^\pm),$ and all products are assumed independent. From Theorem B (i) in \cite{PS} and its symmetric version, the Mellin tranforms of $\ell^\pm$ are given by  
\begin{equation}
\label{MelLT0-}
\E\left[(\ell^-)^{s-1}\right]\; = \; \frac{\sin(\pi \gamma s)}{\sin(\pi (1-\gamma) s)}\qquad \mbox{and}\qquad \E\left[(\ell^+)^{s-1}\right]\; =\;  \frac{\sin(\pi \overline{\gamma} s)}{\sin(\pi (1-\overline{\gamma}) s)}
\end{equation}
for each real $s$ in the respective domain of definition, which is in both cases an open interval containing 1. This entails that $\E[\vert \log(\ell^\pm)\vert] < +\infty,$ with 
\begin{equation}
\label{EspLT0-}
\E\left[\log (\ell^-)\right]\; = \; \pi\cot(\pi\gamma)\, > \, 0 \qquad \mbox{and}\qquad \E\left[\log (\ell^+)\right]\; = \; \pi\cot(\pi\oga)\, > \, 0.
\end{equation}
The following lemma is intuitively obvious. 
\begin{lemma}
\label{infinity}
Assume $(x,y)\in\P_-$. Then one has $\Ton\to +\infty$ and the 
process $Z$ never hits the origin, a.s. under $\pbxy.$
\end{lemma}

\proof To prove the first statement, it is enough to show that $S_n = \Ton - T_0^{(n-1)}\to +\infty$ a.s. as $n\to\infty.$ Set
$$\kar\; =\; \frac{\pi\a}{2}( \cot(\pi\g)\, +\, \cot(\pi\oga))\; =\; \frac{\pi\a\sin(\pi(\g+\oga))}{2\sin(\pi\g)\sin(\pi\oga)}\; >\; 0.$$
From the above discussion, we have
\begin{equation}
\label{S2p}
S_{2p}\; \elaw\;  \vert L_{T_0}\vert^\a\; \times\; \tau^+\, \times\,\lpa  \prod_{k=1}^{p-1} \;\ell_k^-\,\times\,\ell_k^+\rpa^\a
\end{equation}
for every $p\ge 2,$ with independent products on the right-hand side. For every $\eps\in (0, \kar),$ this entails
\begin{eqnarray*}
\pbxy\lcr S_{2p} \le e^{2(p-1)(\kar -\eps)}\rcr & \le & \pbxy \lcr \vert L_{T_0}\vert\le e^{-\eps(p-1)/2\a}\rcr \; +\; \Pb\lcr \tau^+ \le e^{-\eps(p-1)/2}\rcr \\
& & \qquad\qquad + \;\, \Pb\lcr\frac{1}{p-1}\; \sum_{k=1}^{p-1}\;\log(\ell_k^-) \le \pi\cot(\pi\g) -\eps/2\rcr\\
& &  \qquad\qquad\qquad+\;\, \Pb\lcr\frac{1}{p-1}\; \sum_{k=1}^{p-1}\;\log(\ell_k^+) \le \pi\cot(\pi\oga) -\eps/2 \rcr.
\end{eqnarray*}
From Lemma \ref{lem:asympLT0}, there exists $\theta_1(\eps) > 0$ such that $\pbxy \lcr \vert L_{T_0}\vert\le e^{-\eps(p-1)/2\a}\rcr < e^{-p\theta_1(\eps)}$ for $p$ large enough. On the other hand, we have
\begin{eqnarray*}
\Pb\lcr \tau^+ \le e^{-\eps(p-1)/2}\rcr & \le & \Pb_{(0,1)}\lcr \inf\{L_t, \; t \le e^{-\eps(p-1)/2}\} < 0\rcr\\
& = & \Pb_{(0,0)}\lcr \sup\{{\hat L}_t, \; t \le 1\} > e^{\eps(p-1)/2\a}\rcr \; \le \; e^{-p\theta_2(\eps)}
\end{eqnarray*}
for some $\theta_2(\eps) > 0$ and all $p$ large enough, where we have set ${\hat L} = -L,$ the equality following from translation invariance and self-similarity, and the second inequality from the general estimate of Theorem 12.6.1 in \cite{ST}. Last, the existence of some $\theta_3(\eps) > 0$ such that both remaining terms can be bounded from above by $e^{-p\theta_3 (\eps)}$ for $p$ large enough is a standard consequence of (\ref{MelLT0-}), (\ref{EspLT0-}) and Cram\'er's theorem - see e.g. Theorem 1.4 in \cite{dH}, recalling that the assumption (I.5) can be replaced by (I.17)  therein. We can finally appeal to the Borel-Cantelli lemma to deduce, having let $\eps\to 0,$
\begin{equation}
\label{linf}
\liminf_{p\to \infty} \frac{1}{2p} \log (S_{2p})\; \ge \; \kar\; >\; 0\qquad\mbox{a.s.}
\end{equation}
This shows that $S_{2p} \to +\infty$ a.s. and an entirely similar argument yields $S_{2p+1} \to +\infty$ a.s. This concludes the proof of the first part of the lemma. 

The second part is easier. If $\a\le 1,$ it is well-known that $L$ never hits zero, so that $Z$ never hits the origin. If $\a > 1,$ we see from Lemma \ref{lem:asympLT0} that $L_{T_0}$ has no atom at zero under $\pbxy$ and because $\ell^\pm$ are absolutely continuous, all $L_{\Ton}$'s have no atom at zero. We finally get 
$$\pbxy\lcr\mbox{$Z$ visits the origin}\rcr\; =\; \pbxy\lcr \bigcup_{n\ge 1} \left\{ L_{\Ton} = 0\right\}\rcr\; =\; 0$$
where the first identification comes from the fact that $\Ton\to +\infty$ a.s.
\endproof

We can now finish the proof of Theorem A. Set $\theta_0 = \widehat{Z_0Z_{T_0}} \in (-\pi, 0)$ a.s. Observing as in \cite{McK} the a.s. identifications
$$\left\{ \omega(t)\ge -(n-1)\pi +\theta_0\right\}\; =\; \{ \Ton\ge t\}\quad\mbox{and}\quad \left\{ \omega(t)\le -(n-2)\pi +\theta_0\right\}\; =\; \{ T_0^{(n-1)}\le t\},$$
we see that Theorem A amounts to show that
$$\frac{1}{n} \log(\Ton)\; \cas\; \frac{\pi\a\sin(\pi (\gamma+\overline{\gamma}))}{2\sin(\pi\gamma)\sin(\pi \overline{\gamma})} \; =\; \kar \qquad \mbox{as $n\rightarrow +\infty.$}$$
Firstly, with the above notation, we have a.s. under $\pbxy$
$$\liminf_{n\to \infty}\frac{1}{n} \,\log(\Ton)\; \ge\;
\liminf_{n\to \infty} \frac{1}{n} \,\log (S_n)\; \ge \; \kar,$$
where the second inequality comes from (\ref{linf}) and its analogue for $n$ odd. To obtain the upper bound, we will proceed as in the above Lemma \ref{infinity}. Fixing $\varepsilon>0,$ we have
\begin{eqnarray*}
\pbxy\lcr \Ton \ge e^{n(\kar +\eps)}\rcr & \le &  \sum_{k=1}^{n}\; \pbxy\lcr S_k \ge n^{-1} e^{n(\kar +\eps)}\rcr \;\le\;\sum_{k=1}^{n}\; \pbxy\lcr S_k \ge e^{n(\kar +\eps/2)}\rcr
\end{eqnarray*}
for $n$ large enough, with the above notation for $S_k$ and having set $S_1 = T_0.$ Recalling (\ref{S2p}) we have for every $k = 2p\le n$
\begin{eqnarray*}
\pbxy\lcr S_{2p} \ge e^{n(\kar +\eps/2)}\rcr & \le & \pbxy \lcr \vert L_{T_0}\vert\ge e^{n\eps/8\a}\rcr \; +\; \Pb\lcr \tau^+ \ge e^{n\eps/8}\rcr \\
& & \qquad\qquad + \;\, \Pb\lcr\frac{2}{n}\; \sum_{k=1}^{p-1}\;\log(\ell_k^-) \ge \pi\cot(\pi\g) +\eps/4\rcr\\
& &  \qquad\qquad\qquad+\;\, \Pb\lcr\frac{2}{n}\; \sum_{k=1}^{p-1}\;\log(\ell_k^+) \ge \pi\cot(\pi\oga) +\eps/4 \rcr.
\end{eqnarray*}
Again from Lemma \ref{lem:asympLT0}, there exists $\theta_4(\eps) > 0$ such that $\pbxy \lcr \vert L_{T_0}\vert\ge e^{n\eps/8\a}\rcr < e^{-n\theta_4(\eps)}$ for $n$ large enough, whereas
\begin{eqnarray*}
\Pb\lcr \tau^+ \ge e^{n\eps/8}\rcr & \le &  \Pb_{(0,0)}\lcr \sup\{{\hat L}_t, \; t \le 1\} < e^{-n\eps/8\a}\rcr \; \le \; e^{-n\theta_5(\eps)}
\end{eqnarray*}
for some $\theta_5(\eps) > 0$ and $n$ large enough, the second inequality following e.g. from Proposition VIII.2 in \cite{B}. To handle the third term, we separate according as $p\le \sqrt{n}$ or $p > \sqrt{n}.$ In the first case, we have the upper bound
\begin{eqnarray*}
\Pb\lcr\frac{2}{n}\; \sum_{k=1}^{p-1}\;\log(\ell_k^-) \ge \pi\cot(\pi\g) +\eps/4\rcr & \le & \sum_{k=1}^{\sqrt{n}}\; \Pb\lcr\frac{2}{{\sqrt n}}\; \log(\ell_k^-) \ge \pi\cot(\pi\g) +\eps/4\rcr \\
& \le & e^{-\theta_6(\eps)\sqrt{n}}
\end{eqnarray*}
for some $\theta_6(\eps) > 0$ and $n$ large enough, using Lemma \ref{lem:asympLT0} for the second inequality. In the second case, applying Cram\'er's theorem exactly as in Lemma \ref{infinity} gives the upper bound
$$\Pb\lcr\frac{2}{n}\; \sum_{k=1}^{p-1}\;\log(\ell_k^-) \ge \pi\cot(\pi\g) +\eps/4\rcr \; \le\; e^{-\theta_7(\eps)\sqrt{n}}$$
for some $\theta_7(\eps) > 0$ and $n$ large enough. The fourth term is estimated in the same way and we finally get the existence of some $\theta (\eps) > 0$ such that
$$\pbxy\lcr S_{2p} \ge e^{n(\kar +\eps/2)}\rcr\; \le\; e^{-\theta(\eps)\sqrt{n}}$$
for $n$ large enough. An analogous estimate is obtained for $\pbxy\lcr S_{2p+1} \ge e^{n(\kar +\eps/2)}\rcr$ and we can apply as usual the Borel-Cantelli lemma to show the required upper bound
$$\limsup_{n\to \infty}\frac{1}{n}\, \log(\Ton)\; \le \; \kar +\eps,$$
for all $\eps >0,$ a.s. under $\pbxy.$ 

\qed

\subsection{Proof of Theorem B} Recall the decomposition $\Ton = S_1 +\cdots +S_n$ with $S_1\elaw T_0,$
$$S_{2p}\, \elaw\,  \vert L_{T_0}\vert^\a\, \times\, \tau^+\, \times\,\lpa  \prod_{k=1}^{p-1} \,\ell_k^-\,\times\,\ell_k^+\rpa^\a, \quad S_{2p+1}\, \elaw\,  \vert L_{T_0}\vert^\a\; \times\; \tau^-\, \times\,\lpa\prod_{k=1}^{p-1} \;\ell_k^-\,\times\,\prod_{k=1}^{p} \ell_k^+\rpa^\a,$$
and the above notation for $(\tau^\pm, \ell^\pm).$ Let us first investigate the upper tails of the distribution of each $S_k$ under $\pbxy$. We know that
$$\pbxy[T_0 > t]\; \asymp\; \pbxy[\vert L_{T_0}\vert^\a > t]\; \asymp\; \Pb[\tau^- > t]\; \asymp\; \Pb[(\ell^-)^\a > t]\; \asymp\; t^{-\theta}$$
and
$$\Pb[\tau^+ > t]\; \asymp\; \Pb[(\ell^+)^\a > t]\; \asymp\; t^{-\otheta}$$
as $t \to +\infty.$ Supposing first $\rho = 1/2$ viz. $\theta =\otheta,$ a successive application of Lemma \ref{lem:ln} shows that
$$\pbxy [S_k > t] \; \asymp\; t^{-\theta} (\log t)^{k-1}\qquad\mbox{as $t\to +\infty$}$$
for every $k\ge 1.$ Suppose then $\rho < 1/2$ viz. $\theta <\otheta,$ we obtain in a similar way
$$\pbxy [S_{2p} > t] \; \asymp\; t^{-\theta} (\log t)^{p-1}\quad \mbox{and} \quad \pbxy [S_{2p+1} > t] \; \asymp\; t^{-\theta} (\log t)^{p}\quad\mbox{as $t\to +\infty,$}$$
for every $p\ge 1.$ Last, if $\rho > 1/2$ we find
$$\pbxy [S_{2p} > t] \; \asymp\; t^{-\otheta} (\log t)^{p-1}\quad \mbox{and} \quad \pbxy [S_{2p+1} > t] \; \asymp\; t^{-\otheta} (\log t)^{p-1}\quad\mbox{as $t\to +\infty,$}$$
for every $p\ge 1.$ All in all, for all $k \ge 2,$ this shows that
$$\left\{\begin{array}{ll}
\Pb_{(x,y)}[S_k>t]\,\asymp \,t^{-\theta}\; (\log t)^{\left[\frac{k-1}{2}\right]} & \mbox{if $\rho < 1/2$}\\
\vspace{-.2cm}\\
\Pb_{(x,y)}[S_k>t]\,\asymp \,t^{-\otheta}\; (\log t)^{\left[\frac{k}{2}\right]-1} & \mbox{if $\rho > 1/2$}\\
\vspace{-.2cm}\\
\Pb_{(x,y)}[S_k>t]\,\asymp \,t^{-\theta}\; (\log t)^{k-1} & \mbox{if $\rho =1/2,$}
\end{array}\right.$$
and we also know that $\pbxy [S_1 > t]\asymp t^{-\theta}.$ The immediate estimate $\pbxy[\Ton > t] \ge \pbxy[S_n > t]$ yields the required lower bound. To get the upper bounds, it suffices to write 
$$\pbxy[\Ton > t]\; \le\; \sum_{k=1}^n \;\pbxy[S_k > t/n]$$
and to control the sum separately according as $\rho < 1/2, \rho > 1/2$ and $\rho = 1/2.$ We leave the details to the reader.

\qed

\bigskip
 
\noindent
{\bf Acknowledgement.}  Ce travail a b\'en\'efici\'e d'une aide de la Chaire {\em March\'es en Mutation}, F\'ed\'eration Bancaire Fran\c{c}aise.


\begin{thebibliography}{10}

\bibitem{B}
J.~Bertoin. {\em L\'evy processes.} Cambridge University Press, Cambridge, 1996.

\bibitem{BW}
J.~Bertoin and W.~Werner. Stable windings. {\em Ann. Probab.} {\bf 24}, 1269-1279, 1996.

\bibitem{dH}
F.~den Hollander. {\em Large deviations.} Fields institute monographs, AMS, 2000.

\bibitem{DV}
R.~Doney and S.~Vakeroudis. Windings of planar stable processes. {\em S\'em. Probab.} {\bf 45}, 277-300, 2013.

\bibitem{FGD}
P.~Flajolet, X. Gourdon and P.~Dumas. Mellin transforms and asymptotics: Harmonic sums. {\em Theoret. Comput. Sci.} {\bf 144}, 3-58, 1995.

\bibitem{Ko}
A.~N.~Kolmogorov. Zuf\"allige Bewegungen (zur Theorie der Brownschen Bewegung). {\em Ann. Math.} {\bf 35}, 116-117, 1934.
 
\bibitem{La2}
A.~Lachal.  Les temps de passage successifs de l'int\'egrale du mouvement brownien. {\em Ann. Inst. H. Poincar\'e Probab. Statist.} {\bf 33} (1), 1-36, 1997.  
 
\bibitem{McK}
H.~P.~McKean. A winding problem for a resonator driven by a white noise. {\em J. Math. Kyoto Univ.} {\bf 2}, 227-235, 1963.

\bibitem{CP1}
C.~Profeta. Some limiting laws associated with the integrated Brownian motion. To appear in {\em ESAIM Probab. Statist.} Available at {\tt arXiv:1307.1395}

\bibitem{PS}
C.~Profeta and T.~Simon. Persistence of integrated stable processes. Available at {\tt arXiv:1403.1064}.

\bibitem{ST}
G.~Samorodnitsky and M.~S.~Taqqu. {\em Stable Non-Gaussian random processes.} Chapman \& Hall, New-York, 1994. 

\bibitem{Spi}
F.~Spitzer. Some theorems concerning 2-dimensional Brownian motion. {\em Trans. Amer. Math. Soc.} {\bf 131} (4), 733-747, 1956.

\bibitem{Z}
V.~M.~Zolotarev. {\em One-dimensional stable distributions.} Nauka, Moskva, 1983. 

\end{thebibliography}
\end{document}